\documentclass{amsart}
\usepackage{amsmath,amssymb,hyperref,mathrsfs,color}
\usepackage{paralist}

\makeatletter
\def\setliststart#1{\setcounter{\@listctr}{#1}%
  \addtocounter{\@listctr}{-1}}
\makeatother

\usepackage{tikz}
\usetikzlibrary{intersections}
\usepackage{calc}
 \newtheorem{The}{Theorem}[section]
 
 \newtheorem{Lem}[The]{Lemma}
 \newtheorem{Pro}[The]{Proposition}
 \theoremstyle{definition}
 \newtheorem{defn}[The]{Definition}
 \theoremstyle{remark}
 \newtheorem{Rem}[The]{Remark}
 
 \numberwithin{equation}{section}

\newcommand{\R}{\mathbb{R}}

\newcommand{\N}{\mathbb{N}}

\title{Lasry-Lions approximations for discounted Hamilton-Jacobi equations}
\author{Cui Chen \and Wei Cheng \and Qi Zhang}
\address{Faculty of Science, Jiangsu University, Zhenjiang 212013, China}
\curraddr{School of Mathematical Sciences, Fudan University, Shanghai, 200433, China}
\email{chenc@ujs.edu.cn; chencui@fudan.edu.cn}
\address{Department of Mathematics, Nanjing University,
Nanjing 210093, China}
\email{chengwei@nju.edu.cn}
\address{Department of Mathematics, Nanjing University of Aeronautics and Astronautics, Nanjing 210016, China}
\email{zhangqi@nuaa.edu.cn}
\date{}
\subjclass[2010]{26B25, 35A21, 49L25, 37J50, 70H20}
\keywords{Lasry-Lions regularization, Aubry-Mather theory, Hamilton-Jacobi equations, weak KAM theory.}
\begin{document}
\maketitle

\begin{abstract}
	We study the Lasry-Lions approximation using the kernel determined by the fundamental solution with respect to a time-dependent Tonelli Lagrangian. This approximation process is also applied to the viscosity solutions of the discounted Hamilton-Jacobi equations.
\end{abstract}

\section{Introduction}
The method of Lasry-Lions regularization is a kind of variational approximation which is a generalization of the Moreau-Yosida approximation in convex analysis, see, for instance \cite{Lasry-Lions} and \cite{Attouch}. Beyond the analytic aspect of such regularization using the standard kernel $|x-y|/2t$ ($x,y\in\R^n$ and $t>0$), more dynamical aspect of this method has already been studied widely in the past decade, especially with an emphasis on the weak KAM theory and Mather theory: 
\begin{itemize}[--]
  \item An explanation of such a method using the fundamental solution of the associated Hamilton-Jacobi equations instead of the quadratic kernels was first given by Bernard (\cite{Bernard2007}). In the context of weak KAM theory, this method is closely connected to the Lax-Oleinik operators $T^{\pm}_{s,t}$ (\cite{Bernard2008}, \cite{Bernard2012} and \cite{Fathi2012}).
  \item Ilmanen's lemma on insertion of $C^{1,1}$ functions between a semiconvex function less than a semiconcave function (\cite{Bernard2010} and \cite{Fathi-Zavidovique}).
  \item In \cite{Chen-Cheng}, the authors also obtained the limiting behavior of the derivatives of the approximating sequence and the relation between the regular and singular dynamics of the associated Hamiltonian dynamical systems.
  \item There also exists a connection to the theory of generalized characteristics by the recent work on global propagation of singularities of the viscosity solutions of Hamilton-Jacobi equations (\cite{Cannarsa-Cheng3}), and \cite{Cannarsa-Chen-Cheng,Cannarsa-Cheng-Zhang,CCF} for more about the singularities propagation of weak KAM solutions.
  \item An application of standard Lasry-Lions approximation to the minimal homoclinic orbits (\cite{Cannarsa-Cheng2}).
\end{itemize}

Let $H:\R^n\times\R^n\to\R$ be a Tonelli Hamiltonian (i.e., $H=H(x,p)$ is of $C^2$ class and it is strictly convex in $p$ and uniformly superlinear in $p$), and let $L:\R^n\times\R^n\to\R$ be the associated Tonelli Lagrangian. In this paper, we extend the Lasry-Lions regularization procedure to the viscosity solution of the discounted Hamilton-Jacobi equation
$$
\lambda u^{\lambda}(x)+H(x,Du^{\lambda}(x))=0,\quad x\in\R^n
$$
with a discount factor $\lambda>0$. The associated dynamical system is dissipative system and it is a very special kind of contact type Hamiltonian systems (see, for instance, \cite{Wang-Wang-Yan1}, \cite{Wang-Wang-Yan2}, \cite{Cannarsa-Cheng-Yan} and \cite{Zhao-Cheng}). In fact, by defining a new Hamiltonian $H^{\lambda}(t,x,p)=e^{\lambda t}H(x,e^{-\lambda t}p)$, this equation can be reduced to a time-dependent evolutionary Hamilton-Jacobi equation
$$
	D_tv+H^{\lambda}(t,x,D_xv)=0.
$$
Therefore, one can define a kind of {\em intrisic Lasry-Lions regularization} with respect to $u^{\lambda}$ as
$$
\hat{T}_tu^{\lambda}(x)=\sup_{y\in\R^n}\{u^{\lambda}(y)-A^{\lambda}_{0,t}(x,y)\},
$$
where $A^{\lambda}_{0,t}(x,y)$ is the fundamental solution with respect to the time-dependent Lagrangian $L^{\lambda}(t,x,v)=e^{\lambda t}L(x,v)$.

The main result of this paper clarifies the approximation property of this kind of Lasry-Lions regularization. We obtain not only the uniform convergence of $\hat{T}_tu^{\lambda}$ to $u^{\lambda}$ but also the limit of $D\hat{T}_tu^{\lambda}$ as $t\to0^+$. It is worth noting that the latter is closely connected to the intrinsic explanation of the propagation of singularities along generalized characteristics of the associated viscosity solutions (\cite{Cannarsa-Cheng3}).

To study the aforementioned intrinsic approximation, we need the regularity properties of the fundamental solutions $A_{s,t}(x,y)$, which is the least action of the absolutely continuous curve connecting $x$ to $y$ from time $s$ to $t$. The required regularity result is a generalization of the relevant result in \cite{Cannarsa-Cheng3}.

The paper is organized as follows: In section 2, we briefly review some fundamental facts of semiconcave functions and Tonelli's theory in the calculus of variation. In section 3, we prove a Lasry-Lions approximation result for the time-dependent Lagrangians, then discuss this approximation method in a model of discounted Hamilton-Jacobi equations and its connection to the propagation of singularities of associated viscosity solutions.

\medskip

\noindent{\bf Acknowledgments} This work was partially supported by the Natural Scientific Foundation of China (Grant No. 11631006, No. 11501290, No. 11471238), and the  National Basic Research Program of China (Grant No. 2013CB834100). 
\section{viscosity solutions and semiconcave functions}

In this section, we briefly review some basic properties of semiconcave functions and  the viscosity solutions of Hamilton-Jacobi equations.

Let $\Omega\subset\R^n$ be a convex open set. A function $u:\Omega\to\R^n$ is {\em semiconcave} (with linear modulus) if there exists a constant $C>0$ such that
$$
\lambda u(x)+(1-\lambda)u(y)-u(\lambda x+(1-\lambda)y)\leqslant
\frac{C}2\lambda(1-\lambda)|x-y|^2
$$
for any $x,y\in\Omega$ and $\lambda\in[0,1]$. The constant $C$ that satisfies the above inequality is called a semiconcavity constant of $u$ in $\Omega$. A function $u$ is said to be {\em locally semiconcave} if for each $x\in\Omega$ there exists an open ball $B(x,r)\subset\Omega$ such that $u$ is a semiconcave function on $B(x,r)$.

\begin{defn}
Let $u:\Omega\subset\R^n\to\R$ be a continuous function, $x\in\Omega$, the following closed convex sets
\begin{align*}
D^+u(x)=&\{p\in\R^n:\limsup_{y\to x}\frac{u(y)-u(x)-\langle p,y-x\rangle}{|y-x|}\leqslant0\}\\
D^-u(x)=&\{p\in\R^n:\liminf_{y\to x}\frac{u(y)-u(x)-\langle p,y-x\rangle}{|y-x|}\geqslant0\}
\end{align*}
are called the {\em superdifferential} and {\em subdifferential} of $u$ at $x$ respectively.
\end{defn}

\begin{defn}
Let $u:\Omega\subset\R^n\to\R$ be locally Lipschitz. We call a vector $p\in\R^n$ a {\em limiting differential} of $u$ at $x$ if there exists a sequence $\{x_k\}\subset\Omega\setminus\{x\}$ such that $u$ is differentiable at $x_k$ for all $k\in\N$ and
\begin{equation*}
\lim_{k\to\infty}x_k=x,\quad \text{and}\quad \lim_{k\to\infty}Du(x_k)=p.
\end{equation*}
The set of all limiting differentials of $u$ at $x$ is denoted by $D^*u(x)$.
\end{defn}

\begin{Pro}[\cite{Cannarsa-Sinestrari}]
Let $u:\Omega\to\R$ be a semiconcave function and $x\in\Omega$. Then the following properties hold:
\begin{enumerate}[\rm (1)]
\item $D^+u(x)$ is a nonempty closed convex set in $\R^n$ and $D^*u(x)\subset\partial D^+u(x)$, where $\partial D^+u(x)$ denotes the topological boundary of $D^+u(x)$.
\item the set-value function $x\rightsquigarrow D^+u(x)$ is upper semi-continuous.
\item $D^+u(x)=\mbox{\rm co}\,D^*u(x)$.
\item If $D^+u(x)$ is a singleton, then $u$ is differentiable at $x$. Moreover, if $D^+u(x)$ is a singleton for every point in $\Omega$, then $u\in C^1(\Omega)$.
\end{enumerate}
\end{Pro}

Recall that a continuous real-valued function $u$ on $(0,+\infty)\times\R^n$ is called a {\em viscosity supersolution} (resp. {\em viscosity subsolution}) of the Hamilton-Jacobi equation $D_tu+H(t,x,D_xu)=0$ if for any $(t,x)\in (0,+\infty)\times\R^n$
$$
p_t+H(t,x,p_x)\geqslant0\ (\text{resp.}\ \leqslant0), \qquad \forall(p_t,p_x)\in D^-u(t,x)\ (\text{resp.}\ D^+u(t,x)).
$$
A continuous function $u$ is called a {\em viscosity solution} of equation if it is both a viscosity subsolution and a viscosity supersolution.

In this paper, we concentrate on Lagrangians on Euclidean configuration space $\R^n$.
We say that a function $\theta:[0,+\infty)\to[0,+\infty)$ is {\em superlinear} if $\theta(r)/r\to+\infty$ as $r\to+\infty$.

\begin{defn}
A function $L:\R\times\R^n\times\R^n\to\R$ is called a (time-dependent) {\em Tonelli Lagrangian} if $L$ is a function of class $C^2$ satisfying the following conditions:
\begin{enumerate}[(L1)]
\item $L_{vv}(t,x,v)$ is positive definite for all $(t,x,v)\in\R\times\R^n\times\R^n$.
\item There exist two superlinear functions $\theta,\bar{\theta}:[0,+\infty)\to[0,+\infty)$ and a constant $c_0\geqslant 0$ such that
$$
\bar{\theta}(|v|)\geqslant L(t,x,v)\geqslant\theta(|v|)-c_0,\quad(t,x,v)\in\R\times\R^n\times \R^n.
$$
\item There exists a constant $c>0$ such that
$$|L_t(t,x,v)|\leqslant c(1+L(t,x,v)),\quad (t,x,v)\in\R\times\R^n\times\R^n.
$$
\end{enumerate}
\end{defn}


Let $L$ be a Tonelli Lagrangian and let $H$ be the associated Hamiltonian. Given $x\in\R^n$, $y\in B(x,R)$ with $R>0$, and $s<t$, we define
$$
\Gamma^{s,t}_{x,y}=\{{\xi\in W^{1,1}([s,t],\R^n): \xi(s)=x,\xi(t)=y}\},
$$
and
\begin{equation}\label{fundamental_solution}
A_{s,t}(x,y)=\inf_{\xi\in\Gamma^{s,t}_{x,y}}\int^t_sL(\tau,\xi(\tau),\dot{\xi}(\tau))d\tau.
\end{equation}
The existence the minimizers in \eqref{fundamental_solution} is a well known result in Tonelli's theory, (see, for instance, \cite{Cannarsa-Sinestrari}). We call $\xi\in\Gamma^{s,t}_{x,y}$ a {\em minimizer for $A_{s,t}(x,y)$} if
$$
A_{s,t}(x,y)=\int^t_sL(\tau,\xi(\tau),\dot{\xi}(\tau))d\tau.
$$
It is well known that such a minimizer $\xi$ must be of class $C^2$.

It is known that, for any $t_0\in\R$, $x_0\in\R^n$, the function $u(t,x)=A_{t_0,t}(x_0,x)$ is called a {\em fundamental solution} of the Hamilton-Jacobi equation
\begin{equation}\label{eq:HJ_evolution}
	D_tu(t,x)+H(t,x,D_xu(t,x))=0 \qquad x\in\R^n, t>t_0.
\end{equation}
When considering the Cauchy problem with initial condition $u(t_0,x)=u_0(x)$ with $u_0\in\text{Lip}\,(\R^n)$, the associated unique viscosity solution has the following representation:
\begin{equation}\label{eq:Cauchy}
	u(t,x)=\inf_{y\in\R^n}\{u_0(y)+A_{t_0,t}(y,x)\},\quad x\in\R^n, t>t_0.
\end{equation}
Let us recall the Lax-Oleinik operators for time-dependent Lagrangians. For any $s<t$, we define
\begin{align}
T^+_{s,t}u_0(x):=&\sup_{y\in\R^n}\{u_0(y)-A_{s,t}(x,y)\},\label{T^+u}\\
T^-_{s,t}u_0(x):=&\inf_{y\in\R^n}\{u_0(y)+A_{s,t}(y,x)\}.\label{T^-u}
\end{align}
Therefore, $u(t,x)=T^-_{t_0,t}u_0(x)$ is the unique viscosity solution of \eqref{eq:HJ_evolution} with the initial condition $u(t_0,x)=u_0(x)$. For any $t_1>t_0$, it is well known that $u_1(x)=u(t_1,x)=T^-_{t_0,t_1}u_0(x)$ is a locally semiconcave function (see \cite{Cannarsa-Sinestrari}).

\section{Lasry-Lions approximation for discounted equations}

\subsection{Positive type Lax-Oleinik Operators in time-dependent case}
In \cite{Cannarsa-Chen-Cheng}, the authors studied the intrinsic relation between propagation of singularities and the procedure of sup-convolution. In this section, we concentrate on the case of sup-convolution $T^+_{s,t}u$ with $u$ a local semiconcave function.

Let $u:\R^n\to\R$ be a locally semiconcave function. Fixed $x\in\R^n$, $t_0>0$, $\kappa>0$ and $0<T<1$. For any $t\in[t_0,t_0+T]$, we define the {\em local barrier function} $\psi_{t_0,t}^x:\bar{B}(x,\kappa(t-t_0))\to\R$ as
$$
\psi_{t_0,t}^x(y):=u(y)-A_{t_0,t}(x,y).
$$
Now we need the following condition:
\begin{equation}\label{eq:Max}
	\text{$\psi_{t_0,t}^x$ attains a unique maximum point in $B(x,\kappa(t-t_0))$.}\tag{M}
\end{equation}
The following result shows condition \eqref{eq:Max} is satisfied if $u\in\mbox{\rm Lip}\,(\R^n,\R)$. It is a slight generalization of Lemma 3.1 in \cite{Cannarsa-Cheng3}.

\begin{Lem}\label{sup_max}
Suppose $L$ is a Tonelli Lagrangian and let $u_0\in\mbox{\rm Lip}\,(\R^n,\R)$. Then, the supremum in \eqref{T^+u} is attained for every $(t,x)\in(t_0,+\infty)\times\R^n$. Moreover, there exists a  constant $\kappa_0>0$, depending only on $\mbox{\rm Lip}\,(u_0)$\footnote{$\mbox{\rm Lip}\,(u_0)$ stands for the least Lipschitz constant of $u_0$}, such that, for any $(t,x)\in(t_0,+\infty)\times\R^n$ and any maximum point $y_{t,x}$ of $\psi_{t_0,t}^x(y)$, we have
\begin{equation}
\label{eq:maxbound}
	|y_{t,x}-x|\leqslant \kappa_0(t-t_0)\,.
\end{equation}
\end{Lem}

If $\xi_t:[t_0,t]\to\R^n$ is the unique minimizer for $A_{t_0,t}(x,y)$, we define the associated dual arc $p_t$ as
$$
p_t(s)=L_v(s,\xi_t(s),\dot{\xi}_t(s)),\quad s\in[t_0,t].
$$

\begin{The}\label{main result}
Suppose $u:\R^n\to\R$ is a locally semiconcave function and $L$ is a Tonelli Lagrangian. If condition \eqref{eq:Max} is satisfied, then $T^+_{t_0,t}u$ is of class $C^{1,1}_{loc}$ for all $t\in[t_0,t_0+T]$. Moreover, $\lim_{t\to t_0^+}DT^+_{t_0,t}u(x)=q_x$, where $q_x$ is the unique element of $D^+u(x)$ such that
\begin{equation}
H(t_0,x,q_x)=\min_{p\in D^+u(x)}H(t_0,x,p)
\end{equation}
\end{The}

\begin{proof}
Fix $(t_0,x)\in\R\times\R^n$, we have that $\psi_{t_0,t}^x$ attains the maximun at $y_t\in B(x,\kappa(t-t_0))$ for each $t\in[t_0,t_0+T]$ by condition \eqref{eq:Max}. Let $\xi_t\in\Gamma^{t_0,t}_{x,y_t}$ be the minimizer for $A_{t_0,t}(x,y_t)$, by \eqref{eq:diff_A_t_y}, we have
$$
L_v(t,\xi_t(t),\dot{\xi}_t(t))=D_yA_{t_0,t}(x,y_t)\in D^+u(y_t),
$$
since $y_t$ is a maximizer of $\psi_{t_0,t}^x$. Moreover, the family $\{\dot{\xi}_t(\cdot)\}_{t\in(t_0,t_0+T]}$ is equi-Lipschitz by Lemma \ref{sup_max} and Proposition \ref{compactness_condition}.
Let $v_t:=(\xi_t(t)-x)/(t-t_0)$, we obtain
\begin{equation}\label{eq:1}
  \begin{split}
  \left|\frac{\xi_t(t)-x}{t-t_0}-\dot{\xi}_t(t_0)\right|&\leqslant
  \frac1{t-t_0}\int_{t_0}^t|\dot{\xi}_t(s)-\dot{\xi}_t(t_0)|ds\\
  &\leqslant\frac{C_1}{t-t_0}\int_{t_0}^t(s-t_0)ds
  =\frac{C_1}2(t-t_0).
  \end{split}
\end{equation}
Thus, we have
$$
v_0:=\lim_{t\to t^+_0}v_t=\lim_{t\to t^+_0}\dot{\xi}_t(t_0).
$$
Since $u$ is a locally semiconcave function, for any $y\in B(x,\kappa(t-t_0))$, $p_x\in D^+u(x)$ and $p_y\in D^+u(y)$, we have (\cite[Proposition 3.3.10]{Cannarsa-Sinestrari})
$$
\langle p_y-p_x,y-x\rangle\leqslant C_2|y-x|^2.
$$
Taking any $t_k\to t_0$, we have
$$
p_{y_{t_k}}=L_v(t_k,\xi_{t_k}(t_k),\dot{\xi}_{t_k}(t_k))\in D^+u(y_{t_k}).
$$
Then, for any $p_{x}\in D^+u(x)$
$$
\langle p_x-L_v(t_k,\xi_{t_k}(t_k),\dot{\xi}_{t_k}(t_k)),v_{t_k}\rangle+C_2(t_k-t_0)|v_{t_k}|^2\geqslant 0.
$$
Taking the limit in the above inequality as $k\to\infty$ we obtain
$$
\langle p_{x},v_{0}\rangle\geqslant\langle L_v(t_0,x,v_{0}),v_{0}\rangle=\langle q_x,v_{0}\rangle,\quad \forall p_{x}\in D^+u(x),
$$
where $q_x:=L_v(t_0,x,v_{0})\in D^+u(x)$ by the upper semicontinuity of  $x\rightsquigarrow D^+u(x)$. Thus, for all $p_{x}\in D^+u(x)$,
$$
H(t_0,x,p_{x})\geqslant\langle L_v(t_0,x,v_{0}),v_{0}\rangle-L(t_0,x,v_{0})=H(t_0,x,q_{x}),
$$
and $q_x$ is the unique minimum point of $H(t_0,x,\cdot)$ on $D^+u(x)$.
The uniqueness of $p_x$ implies the uniqueness of $v_{0}$ since $L_v(t_0,x,\cdot)$ is injective, and we have
$$
\lim_{t\to t_0^+}DT^+_{t_0,t}u(x)=\lim_{t\to t_0^+}L_v(t_0,\xi_t(t_0),\dot{\xi}_t(t_0))=q_x.
$$
This completes the proof of the theorem.
\end{proof}

\subsection{Lasry-Lions regularization on discounted equations}

For $\lambda>0$, we consider the Hamilton-Jacobi equations with discount factors,
\begin{equation}\label{eq:discount}\tag{HJ$_d$}
\lambda u^{\lambda}(x)+H(x,Du^{\lambda}(x))=0,\quad x\in\R^n.
\end{equation}
Multiplying $e^{\lambda t}$ in \eqref{eq:discount} and defining $v^{\lambda}(t,x)=e^{\lambda t}u^{\lambda}(x)$, one can check that $v=v^{\lambda}$ is a viscosity solution of
\begin{equation}\label{eq:HJ_evol}\tag{HJ$_e$}
	D_tv+H^{\lambda}(t,x,D_xv)=0
\end{equation}
with $H^{\lambda}(t,x,p)=e^{\lambda t}H(x,e^{-\lambda t}p)$, if $u^{\lambda}$ is a viscosity solution of \eqref{eq:discount}. The associated Lagrangian $L^{\lambda}$ with respect to $H^{\lambda}$ has the form
$$
L^{\lambda}(t,x,v)=e^{\lambda t}L(x,v).
$$

\begin{Pro}
$u^{\lambda}(x)$ is a viscosity solution of \eqref{eq:discount} if and only if $v^{\lambda}(t,x)$  is a viscosity solution of \eqref{eq:HJ_evol}.
\end{Pro}

\begin{proof}
	It is not hard to check that $u^{\lambda}$ is a locally semiconcave function if and and only if so is $v^{\lambda}$ when restricted to any compact time interval. The local semiconcavity properties of viscosity solutions of \eqref{eq:discount} and \eqref{eq:HJ_evol} are well known results, see, for instance, \cite{Cannarsa-Sinestrari}. Thus, our conclusion is a direct consequence of Proposition 5.3.1 in \cite{Cannarsa-Sinestrari}.
\end{proof}

Now, one can define a kind of {\em intrisic Lasry-Lions regularization} with respect to $u^{\lambda}$ as follows: let $u^{\lambda}$ be a viscosity solution of the discounted Hamilton-Jacobi equation \eqref{eq:discount}, define
$$
\hat{T}_tu^{\lambda}(x)=\sup_{y\in\R^n}\{u^{\lambda}(y)-A^{\lambda}_{0,t}(x,y)\}=T^+_{0,t}v^{\lambda}(0,x),
$$
where $A^{\lambda}_{0,t}(x,y)$ is the fundamental solution with respect to the Lagrangian $L^{\lambda}$.

\begin{The}\label{discount_LL}
If $u^{\lambda}$ is a viscosity solution of the discounted Hamilton-Jacobi equation \eqref{eq:discount}, and $\hat{T}_tu^{\lambda}$ is the associated intrinsic Lasry-Lions regularization, then we have $\hat{T}_tu^{\lambda}$ is of class $C^{1,1}_{loc}$ and $\hat{T}_tu^{\lambda}$ tends to $u^{\lambda}$ uniformly as $t\to0^+$. Moreover, there exists an unique $q^{\lambda}_x\in D^+u^{\lambda}(x)$ such that
\begin{equation}\label{eq:w_x}
	H(x,q^{\lambda}_x)=\min_{p\in D^+u^{\lambda}(x)}H(x,p)
\end{equation}
and $\lim_{t\to0^+}D\hat{T}_tu^{\lambda}(x)=q^{\lambda}_x$.
\end{The}

\begin{proof}
	Notice that $L^{\lambda}$ satisfies conditions (L1)-(L3) for any fixed $\lambda>0$. The $C^{1,1}$ regularity of $\hat{T}_tu^{\lambda}$ is a direct consequence of the $C^{1,1}$ regularity of $A_{0,t}(x,\cdot)$, \eqref{eq:diff_A_t_x} and condition \eqref{eq:Max} which holds by a slight generalization of Lemma 3.1 in \cite{Cannarsa-Cheng3}, since $v^{\lambda}(0,\cdot)=u^{\lambda}$ is semiconcave and Lipschitz. Now, fix $T>0$ as in Theorem \ref{main result}, then for any $t\in(0,T]$, there exists a unique maximizer $y_{t,x}$ of $u^{\lambda}(\cdot)-A^{\lambda}_{0,t}(x,\cdot)$, and $\lim_{t\to0^+}y_{t,x}=x$, therefore $\hat{T}_tu^{\lambda}$ tends to $u^{\lambda}$ uniformly as $t\to0^+$.
	
	Applying Theorem \ref{main result} to the solution $v^{\lambda}$ of \eqref{eq:HJ_evol}, there exists a unique $q^{\lambda}_x\in D^+_xv^{\lambda}(0,x)=D^+u^{\lambda}(x)$ such that
$$
\lim_{t\to t_0^+}DT^+_{0,t}v^{\lambda}(0,x)=q^{\lambda}_x\in D^+_xv^{\lambda}(0,x)=D^+u^{\lambda}(x)
$$
and
$$
H^{\lambda}(0,x,q^{\lambda}_x)=\min_{p\in D^+_xv^{\lambda}(0,x)}H^{\lambda}(0,x,p).
$$
which is equivalent to \eqref{eq:w_x}.
\end{proof}

Let $\lambda>0$, a recent work by Davini, {\em et al} (\cite{DFIZ}) shows the unique solution $u^{\lambda}$ of \eqref{eq:discount} converges uniformly, as $\lambda\to0^+$, to a certain viscosity solution of the stationary Hamilton-Jacobi equation
\begin{equation}\label{eq:static}\tag{HJ$_s$}
	H(x,Du(x))=0,
\end{equation}
when $0$ is Ma\~n\'e's critical value. Comparing to the results in \cite{Chen-Cheng}, there exists a unique $q_x\in D^+u(x)$ such that
\begin{equation}\label{eq:q_x}
	H(x,q_x)=\min_{p\in D^+u(x)}H(x,p),
\end{equation}
$\lim_{t\to0^+}DT^+_tu(x)=q_x$. Therefore, one can raise the following problem:

\medskip

\noindent{\bf Problem}: For $q^{\lambda}_x$ and $q_x$ defined in \eqref{eq:w_x} and \eqref{eq:q_x}, does $\lim_{\lambda\to0^+}q^{\lambda}_x=q_x$?

\begin{Rem}
	To answer Problem above, a possible systematic approach will be based on the recent works \cite{Wang-Wang-Yan1} and \cite{Wang-Wang-Yan2}. Moreover, one can understand such a problem as follows (\cite{Cannarsa-Cheng-Yan} and \cite{Zhao-Cheng}): We suppose $L$ is a function of $C^2$ class and it satisfies the following conditions:
\begin{enumerate}[(L1)]
  \item $L_{vv}(x,r,v)>0$ for all $(x,r,v)\in \R^n\times\R\times\R^n$;
  \item For each $r\in\R$, there exist two superlinear and nondecreasing function $\overline{\theta}_r,\theta_r:[0,+\infty)\to[0,+\infty)$, $\theta_r(0)=0$ and $c_r>0$, such that 
  $$
  \overline{\theta}_r(|p|)\geqslant L(x,r,v)\geqslant\theta_r(|p|)-c_r,\quad (x,v)\in\R^n\times\R^n.
  $$
  \item There exists $K>0$ such that
  $$
  |L_r(x,r,v)|\leqslant K,\quad (x,r,v)\in \R^n\times\R\times\R^n.
  $$
\end{enumerate}
Fix $x,y\in\R^n$, $u\in\R$ and $t>0$. Define $\Gamma^t_{x,y}=\{{\xi\in AC([0,t],\R^n): \xi(0)=x,\xi(t)=y}\}$. Let $\xi\in\Gamma^t_{x,y}$, we consider the Carath\'eodory equation 
\begin{equation}\label{eq:caratheodory_L}
	\dot{u}_{\xi}(s)=L(\xi(s),u_{\xi}(s),\dot{\xi}(s)),\quad a.e.\ s\in[0,t]
\end{equation}
with initial conditions $u_{\xi}(0)=u$. We define
\begin{equation}\label{eq:fundamental_solution}
	A(t,x,y,u)=u+\inf_{\xi}\int^t_0L(\xi(s).u_{\xi}(s),\dot{\xi}(s))\ ds,
\end{equation}
where the infinmum is taken over of $\xi\in\Gamma^t_{x,y}$ and $u_{\xi}:[0,t]\to\R^n$ is a absolutely continuous curve determined by \eqref{eq:caratheodory_L}. In the case of discounted equations, $L(x,u,v)=-\lambda u+L(x,v)$. One can define the negative type Lax-Oleinik operator  $T^-_t:C(\R^n,\R)\to C(\R^n,\R)$ for any $t>0$:
\begin{equation}\label{eq:LL-}
	(T^-_t\phi)(x)=\inf_{y\in\R^n}A(t,y,x,\phi(y)).
\end{equation}
It is not very difficult to show the fundamental solution $A(t,x,y,u)$ is locally semiconcave with constants depending on $|x-y|/t$ and $\lambda$. Thus, the key point of the uniform semiconcavity of $u^{\lambda}$ is to show that there exists $\kappa_0>0$ independent of $x$ such that the minimum of $A(t,\cdot,x,\phi(\cdot))$ on $\R^n$ is attained and all the minimum points is contained in $B(x,\kappa_0t)$. If $\phi=u^{\lambda}$ is a Lipschitz weak KAM solution of the equation $H(x,u(x),Du(x))=0$ with respect to $H(x,u,p)=\lambda u+H(x,p)$, then this gives a uniform bound of the velocity all of backward calibrated curves which leads to the uniform constants in the associated semiconcavity estimate. We will answer Problem above in a much more general context in the future.
\end{Rem}

\subsection{Connection to singularities} The intrinsic Lasry-Lions regularization is closely connected to the propagation of singularities of the solution $u^{\lambda}$ of \eqref{eq:discount}. It is obvious that $u^{\lambda}$ shares the singularities of $v^{\lambda}$ in \eqref{eq:HJ_evol}. In this section, we suppose that $0$ is Ma\~n\'e's critical value.

Recall that a point $(t,x)\in\R\times\R^n$ is called a {\em singular point} of a semiconcave function $u(t,x)$ if $D^+u(t,x)$ is not a singleton. The set of all singular points of $u$ is denoted by $\text{Sing}\,(u)$. It is obvious that $(t,x)\in\text{Sing}\,(v^{\lambda})$ if and only if $x\in\text{Sing}\,(u^{\lambda})$.

Using the representation formula of $v^{\lambda}$ (see, for instance, \cite[Proposition 3.5]{DFIZ}), for any $\tau\in\R$, we obtain that
\begin{equation}
	v^{\lambda}(\tau,x)=e^{\lambda \tau}u^{\lambda}(x)=\inf_{\gamma}\int^{\tau}_{-\infty}L^{\lambda}(s,\gamma(s),\dot{\gamma}(s))\ ds,
\end{equation}
where the infimum is taken over all absolutely continuous curves $\gamma:(-\infty,\tau]\to\R^n$, with $\gamma(\tau)=x$. Moreover, there exists a Lipschitz and $C^2$ curve $\gamma_x:(-ˆ'\infty,\tau]\to\R^n$, with $\gamma_x(\tau)=x$, such that, for any $t>\tau$,
\begin{equation}\label{eq:calibration}
	v^{\lambda}(\tau,x)=v^{\lambda}(\tau-t,\gamma_x(\tau-t))+\int^{\tau}_{\tau-t}L^{\lambda}(s,\gamma_x(s),\dot{\gamma}_x(s))\ ds.
\end{equation}
It is clear that $v^{\lambda}$ is differentiable at $(\tau-t,\gamma_x(\tau-t))$ for all $t>\tau$, and $\gamma_x$ is an extremal of the associated Euler-Lagrange equation with respect to $L^{\lambda}$. As in the classical weak KAM theory, it is not difficult to prove that {\em $(\tau,x)$ is a differentiable point of $v^{\lambda}$ if and only if there exists a unique $\gamma_x$ satisfying \eqref{eq:calibration}} (see also \cite[Theorem 6.4.9]{Cannarsa-Sinestrari}).

\begin{The}
	Let $x_0\in\mbox{\rm Sing}\,(u^{\lambda})$, then any maximizer $y_{t,x_0}$ of $v^{\lambda}(t_0,\cdot)-A^{\lambda}_{0,t}(x_0,\cdot)$ is contained in $\mbox{\rm Sing}\,(u^{\lambda})$ for all $t>0$ and there exists $t_1>0$ such that the map $t\mapsto y_{t,x_0}$, the maximizers with respect to $v^{\lambda}(t_0,\cdot)-A^{\lambda}_{0,t}(x_0,\cdot)$ for $0<t<t_1$, is continuous. Moreover, the right derivative $\frac d{dt}y_{t,x_0}\vert_{t=0^+}$ exists and it is equal to $q^{\lambda}_{x_0}$ as in Theorem \ref{discount_LL}, i.e., $q^{\lambda}_{x_0}$ is the unique element in $D_y^+v^{\lambda}(t_0,x_0)$ such that
\begin{equation}\label{eq:w_x1}
	H(x_0,q^{\lambda}_{x_0})=\min_{p\in D_x^+v^{\lambda}(t_0,x_0)}H(x_0,p).
\end{equation}
\end{The}

\begin{proof}
	Fix $t_0\in\R$, then $(t_0,x_0)\in\mbox{\rm Sing}\,(v^{\lambda})$ since $x_0\in\mbox{\rm Sing}\,(u^{\lambda})$. For any $t>0$ and $y_{t,x_0}\in\arg\max\{v^{\lambda}(t_0,\cdot)-A^{\lambda}_{0,t}(x_0,\cdot)\}$ (which is nonempty since $v^{\lambda}(t_0,\cdot)$ is Lipschitz and Lemma \ref{sup_max}), suppose $y_{t,x_0}$ is a differentiable point of $v^{\lambda}(t_0,\cdot)$. Thus 
	$$
	0\in D^+\{v^{\lambda}(t_0,\cdot)-A^{\lambda}_{0,t}(x_0,\cdot)\}(y_{t,x_0})=D_yv^{\lambda}(t_0,y_{t,x_0})-D^-\{A^{\lambda}_{0,t}(x_0,\cdot)\}(y_{t,x_0}),
	$$
	equivalently, $D_yv^{\lambda}(t_0,y_{t,x_0})\in D^-\{A^{\lambda}_{0,t}(x_0,\cdot)\}(y_{t,x_0})$. It follows that $A^{\lambda}_{0,t}(x_0,\cdot)$ is differentiable at $y_{t,x_0}$ and
	$$
	p_{t,x_0}=D_yv^{\lambda}(t_0,y_{t,x_0})=D_yA^{\lambda}_{0,t}(x_0,y_{t,x_0})
	$$
	since $A^{\lambda}_{0,t}(x_0,\cdot)$ is locally semiconcave (see, for instance, \cite{Cannarsa-Sinestrari}). Therefore, there exists two $C^2$ curves $\xi_{t,x_0}:[0,t]\to\R^n$ and $\gamma_{x_0}:(-\infty,t]\to\R^n$ such that $\xi_{t,x_0}(0)=x_0$, $\gamma_{x_0}(t)=\xi_{t,x_0}(t)=y_{t,x_0}$ and $p_{t,x_0}=L_v(\gamma_{x_0}(t),\dot{\gamma}_{x_0}(t))=L_v(\xi_{t,x_0}(t),\dot{\xi}_{t,x_0}(t))$. Since $\xi_{t,x_0}$ and $\gamma_{x_0}$ has the same endpoint condition at $t$, then they coincide on $[0,t]$. Thus, $x_0=\gamma_{x_0}(0)$ and $(0,x_0)$ is a differentiable point of $v^{\lambda}$ since $\gamma_{x_0}$ is a backward calibrated curve by \eqref{eq:calibration}. On the other hand, $(0,x_0)$ is contained in $\mbox{\rm Sing}\,(v^{\lambda})$ since $(t_0,x_0)\in\mbox{\rm Sing}\,(v^{\lambda})$. This leads to a contradiction.
	
	To prove \eqref{eq:w_x1}, we need a slight modification of Theorem \ref{main result}. Notice that $v^{\lambda}(t,\cdot)$ is equi-Lipschitz and equi-semiconcave for $t\in[0,t_1]$. By the regularity properties of the fundamental solutions,  $v^{\lambda}(t_0,\cdot)-A^{\lambda}_{0,t}(x_0,\cdot)$ is strictly concave for $t\in(0,t_2]$, where $t_2\leqslant t_1$ is determined by Proposition \ref{convexity_A_t} and the semiconcavity of $u^{\lambda}$. Therefore, $t\mapsto y_{t,x_0}$ is a continuous selection since the function $v^{\lambda}(t_0,\cdot)-A^{\lambda}_{0,t}(x_0,\cdot)$ is continuos. By the same argument as in the proof of Theorem \ref{main result}, we obtain that
	$$
	\frac d{dt}y_{t,x_0}\vert_{t=0^+}=q^{\lambda}_{x_0}
	$$
	with $q^{\lambda}_{x_0}$ satisfying \eqref{eq:w_x1}.
\end{proof}

\appendix

\section{Regularity properties of fundamental solutions}
Here we collect some relevant regularity results with respect to the fundamental solutions of \eqref{eq:HJ_evolution} on $\R^n$. The proofs of these regularity results are similar to those in \cite{Cannarsa-Cheng3} in the time-independent case. The difference is that that we need an extra condition (L3) to ensure the uniform Lipschitz estimate of the minimizers in the relevant Tonelli-like variational problem. We omit the proof.

\begin{Pro}\label{Main_bound_Lem}
Let $a\leqslant s<t\leqslant b,R>0$ and  suppose $L$ satisfies condition {\rm (L1)-(L3)}. Given  any $x\in\R^n$ and $y\in \overline{B}(x,R)$, let $\xi\in\Gamma^{s,t}_{x,y}$ be a minimizer for $A_{s,t}(x,y)$ and let $p(\cdot)$ be the dual arc. Then we have that
\begin{align*}
	\sup_{\tau\in[s,t]}|\dot{\xi}(\tau)|\leqslant \kappa_T(R/(t-s)),\quad \sup_{\tau\in[s,t]}|p(\tau)|\leqslant \kappa_T(R/(t-s))
\end{align*}
and
\begin{equation*}
\sup_{\tau\in[s,t]}|\xi(\tau)-x|\leqslant\kappa_T(R/(t-s)),
\end{equation*}
where $\kappa_T:(0,\infty)\to(0,\infty)$ is nondecreasing and $T=b-a$.
\end{Pro}

Fix $x\in\R^n$ and suppose $R>0$ and $L$ is a Tonelli Lagrangian.  For any $a\leqslant s<t\leqslant b$, $T=b-a$ and $y\in\overline{B}(x,R)$, let $\xi \in\Gamma^{s,t}_{x,y}$ be a minimizer for $A_{s,t}(x,y)$ and let $p$ be its dual arc. Then there exists a nondecreasing function $\kappa_T:(0,\infty)\to(0,\infty)$ such that
\begin{equation*}
\label{eq:main_bound_not}
\sup_{\tau\in[s,t]}|\dot{\xi}(\tau)|\leqslant\kappa_T(R/(t-s)),\quad
	\sup_{\tau\in[s,t]}|p(\tau)|\leqslant\kappa_T(R/(t-s)),
\end{equation*}
by Proposition \ref{Main_bound_Lem}. Now, $a<b$, $x\in\R^n$ and $\lambda>0$  define  compact sets
\begin{equation*}
	\begin{split}
	\mathbf{K}_{a,b,x,\lambda}&:=[a,b]\times\overline{B}(x,\kappa(4\lambda))\times\overline{B}(0,\kappa(4\lambda))\subset\R\times\R^n\times\R^n,\\
	\mathbf{K}^*_{a,b,x,\lambda}&:=[a,b]\times\overline{B}(x,\kappa(4\lambda))\times\overline{B}(0,\kappa(4\lambda))\subset\R\times\R^n\times(\R^n)^*.
	\end{split}
\end{equation*}
The following  is one of the key technical points.

\begin{Pro}\label{compactness_condition}
Suppose $L$ is a Tonelli Lagrangian. Fix $x\in\R^n$, $\lambda>0$, $s<t$, $T=t-s<1$ and $y\in B(x,\lambda T)$. Let $z\in\R^n$ and $h\in\R$ be such that
\begin{equation*}
|z|<\lambda T\qquad\mbox{and}\qquad s-\frac {T}2<h<1-T.
\end{equation*}
Then  any minimizer
$\xi\in\Gamma^{s,t+h}_{x,y+z}$ for $A_{s,t+h}(x,y+z)$ and  corresponding  dual arc $p$ satisfy the following inclusions
\begin{align*}
	\{(\tau,\xi(\tau),\dot{\xi}(\tau)):\tau\in[s,t+h]\}&\subset \mathbf{K}_{s,s+1,x,\lambda},\\
	\{(\tau,\xi(\tau),p(\tau)):\tau\in[s,t+h]\}&\subset \mathbf{K}^*_{s,s+1,x,\lambda}.
\end{align*}
\end{Pro}

\begin{Pro}\label{semiconcave_A_t}
Suppose $L$ is a Tonelli Lagrangian. Then for any $\lambda>0$ there exists a constant $C_\lambda>0$ such that for any $x\in\R^n$, $s<t$ with $T=t-s<2/3$, $y\in B(x,\lambda T)$, and $(h,z)\in\R\times\R^n$  satisfying $|h|<T/2$ and $|z|<\lambda T$ we have
\begin{equation}\label{eq:seminconcavity_A_t}
A_{s,t+h}(x,y+z)+A_{s,t-h}(x,y-z)-2A_{s,t}(x,y)\leqslant\frac {C_\lambda}T\big(|h|^2+|z|^2\big).
\end{equation}
Consequently, $(t,y)\mapsto A_{s,t}(x,y)$ is locally semiconcave in $(0,1)\times\R^n$, uniformly with respect to $x$ and $s$.
\end{Pro}

\begin{Pro}\label{convexity_A_t}
Suppose $L$ is a Tonelli Lagrangian and, for any $\lambda>0$, there exists  $T_\lambda'>0$such that for any $x\in\R^n$, $s<t$, the function $(t,y)\mapsto A_{s,t}(x,y)$ is semiconvex on the  cone
\begin{equation}
\label{cone}
S_\lambda(x,T_\lambda'):=\big\{(t,y)\in\R\times\R^n~:~T=t-s<T_\lambda',\; |y-x|<\lambda T\big\}\,,
\end{equation}
and there exists a constant $C''_\lambda>0$ such that for all $(t,y)\in S_\lambda(x,T_\lambda')$,  all $h\in[0,T/2)$, and  all $z\in  B(0,\lambda T)$ we have that
\begin{equation}\label{eq:semiconvexity}
A_{s,t+h}(x,y+z)+A_{s,t-h}(x,y-z)-2A_{s,t}(x,y)\geqslant - \frac{C''_ \lambda}{T}(h^2+|z|^2).
\end{equation}
Moreover, there exists $T''_\lambda\in(0,t_\lambda']$ and  $C'''_{\lambda}>0$ such that for all $T\in(0,T''_\lambda]$ the function $A_{s,t}(x,\cdot)$ is uniformly convex on $B(x,\lambda T)$ and   for all $y\in B(x,\lambda T)$ and   $z\in  B(0,\lambda T)$ we have that
\begin{equation}\label{eq:convexity_local}
A_{s,t}(x,y+z)+A_{s,t}(x,y-z)-2A_{s,t}(x,y)\geqslant \frac{C'''_{\lambda}}{T}|z|^2.
\end{equation}
\end{Pro}

\begin{Pro}\label{C11_A_t}
Suppose $L$ is a Tonelli Lagrangian and, for any $\lambda>0$, there exists  $T_\lambda'>0$such that for any $x\in\R^n$ the functions $(t,y)\mapsto A_{s,t}(x,y)$ and $(t,y)\mapsto A_{s,t}(y,x)$ are of class $C^{1,1}_{\text{loc}}$ on the cone $S_{\lambda}(x,T_\lambda')$ defined in \eqref{cone}.
Moreover, for all $(t,y)\in S(x,T_\lambda')$
\begin{align}
D_yA_{s,t}(x,y)=&L_v(t,\xi(t),\dot{\xi}(t)),\label{eq:diff_A_t_y}\\
D_xA_{s,t}(x,y)=&-L_v(s,\xi(s),\dot{\xi}(s)),\label{eq:diff_A_t_x}
\end{align}
where $\xi\in\Gamma^{s,t}_{x,y}$ is the unique minimizer for $A_{s,t}(x,y)$.
\end{Pro}

\end{document}